\documentclass[12pt]{amsart}
\usepackage[left=3cm,right=3cm,top=4cm,bottom=4cm]{geometry}
\usepackage{latexsym}
\usepackage{amsmath}
\usepackage{amsfonts}
\usepackage{amssymb}
\usepackage{hyperref}
\usepackage{tikz}
\usetikzlibrary{shapes.geometric, positioning}
\usepackage{tabularx}
\usepackage{booktabs}
\usepackage{ltablex}
\keepXColumns

\newtheorem{theorem}{Theorem}[section]

\newtheorem{corollary}[theorem]{Corollary}

\theoremstyle{definition}
\newtheorem{definition}[theorem]{Definition}
\newtheorem{example}[theorem]{Example}
\theoremstyle{remark}
\newtheorem{remark}[theorem]{Remark}
\numberwithin{equation}{section}

\newcommand{\CC}{{\mathbb C}}
\newcommand{\QQ}{{\mathbb Q}}

\hypersetup{
  pdftitle={},
  pdfauthor={},
  pdfsubject={},
  pdfkeywords={},
  colorlinks=true,
  breaklinks=true,
  bookmarksopen=true,
  bookmarksnumbered=true,
  pdfpagemode=UseOutlines,
  plainpages=false,
  linkcolor=black,
  citecolor=black,
  anchorcolor=black,
  urlcolor  = black
  }
\newcommand\blfootnote[1]{%
  \begingroup
  \renewcommand\thefootnote{}\footnote{#1}%
  \addtocounter{footnote}{-1}%
  \endgroup
}
\begin{document}

\title{Matroids arisen from seeds}
\author{Fayadh Kadhem}
\address{School of Professional Studies\\
Bahrain Polytechnic\\
Isa Town, Bahrain}
\email{fayadh.kadhem@polytechnic.bh}
\blfootnote{Key words: matroid, cluster algebra, seed, cluster matroid.}
\blfootnote{School of Professional Studies, Bahrain Polytechnic, Isa Town, Bahrain.}
\subjclass{05B35,13F60}
\date{\today}
%\date{July 4, 1776}

\begin{abstract}
This study aims to shed light on new (sub)classes of matroids originating from cluster algebras and investigate their properties. We focus on what we call \textit{cluster matroids} and build some results on them. Then, we point out a relationship between these kinds of matroids and uniform matroids and study their minors.
\end{abstract}

\maketitle

\section{Introduction}
\label{introduction}
The theory of matroids originated in 1935 by Whitney to abstract the notion of the linear independence of vector spaces. Since then, because of its interesting properties and applications, the matroid theory has formed one of the most active areas of algebraic combinatorics. On the other hand, cluster algebras were invented by Fomin and Zelevinsky in 2002 and have quickly received a lot of interest because of their significant applications and connections to different areas of mathematics. For instance, the applications of the theory of cluster algebras appear in representation theory, combinatorics, algebraic geometry, Poisson geometry, integrable systems, mathematical physics, and topology. Although both theories have the flavor of algebra and combinatorics and despite the existence of some works built on both theories, their immediate relationships have not yet been studied well.\\
In this study, we provide an overview of matroids in Section 2 and cluster algebras in Section 3. We then introduce some connections between them in Section 4. Next, we take a close look at the behavior of the minors of the main class of matroids in this study in Section 5.

\section{Matroid preliminaries}

This section introduces the theory of matroids and the needed results. The reader who is interested in a deeper look is referred to Oxley's book \cite{O}.
\begin{definition}
A \textit{matroid} $M$ is a pair $(E, \mathcal{B})$ where $E$ is a finite nonempty set called the \textit{ground set} and $\mathcal{B}$ is a subset of the power set of $E$ in which
\begin{enumerate}
    \item $\mathcal{B} \neq \emptyset$.
    \item If $B_1,B_2 \in \mathcal{B}$ and $x \in B_1 \setminus B_2$, then there is a $y \in B_2 \setminus B_1$ such that $(B_1 \setminus \{x\}) \cup \{y\} \in \mathcal{B}$.
\end{enumerate}
A member of $\mathcal{B}$ is called a \textit{basis} of the matroid $M$, while a subset $I$ of $E$ is called \textit{independent} if it is a subset of a basis. Any subset of $E$ that is not independent is called \textit{dependent}. Sometimes, we may write $\mathcal{B}(M)$ instead of $\mathcal{B}$ to emphasize that we are considering the set of bases of the matroid $M$.
\end{definition}

\begin{example}
Let $A$ be a matrix and $E$ be the set of column labels of $A$. Let $\mathcal{B}$ be the set of linearly independent sets of maximal size induced by the column labels of $A$. Then, $(E, \mathcal{B})$ is a matroid. A matroid is called \textit{representable} if it can be formed by the linear independence relations of a matrix.
\end{example}

\begin{remark}
It is not hard to see that the definition of a matroid is a generalization of the properties of bases of a linear space $V$. A matroid can be defined in other equivalent ways, such as the independent sets or \textit{circuits}, which are the minimal dependent sets.
\end{remark}

\begin{remark}
Let $M$ be a matroid whose ground set is $E$ and set of bases is $\mathcal{B}$. The pair $M^*=(E,\mathcal{B}^*)$, in which $\mathcal{B}^*=\{E \setminus B \mid B \in \mathcal{B} \}$, forms a matroid whose set of bases is $\mathcal{B}^*$. This matroid is called the \textit{dual matroid} of $M$. Bases, (in)dependent sets and circuits in the dual matroid are called cobases, co(in)dependent sets and cocircuits of the original matroid, respectively.
\end{remark}

Now, we recall the following definition:

\begin{definition}
Let $\mathbb{K}$ be an extension field of a field $\mathbb{F}$. Let $\alpha \in \mathbb{K}$ and let $E \subset \mathbb K$.
\begin{enumerate}
\item An element $\alpha$ is said to be \textit{algebraic} over $\mathbb{F}$ if there exists a nonzero polynomial $p \in \mathbb{F}[x]$ such that $p(\alpha)=0$. If no such polynomial exists, $\alpha$ is called \textit{transcendental} over $\mathbb{F}$.

\item A set $E$ is said to be \textit{algebraically independent} over $\mathbb{F}$ if there is no nonzero polynomial $p \in \mathbb{F}[x_1, x_2, \dots, x_n]$, where $n = |E|$, such that $p$ vanishes when evaluated at the elements of $E$. Otherwise, $E$ is said to be \textit{algebraically dependent} over $\mathbb{F}$.

\end{enumerate}
\end{definition} 

\begin{theorem}\label{matroid structure}
Let $\mathbb{K}$ be an extension field of a field $\mathbb{F}$ and $E \subset \mathbb{K}$ be finite. The collection $\mathcal{I}$ of subsets of $E$ that are algebraically independent over $\mathbb{F}$ forms a matroid on $E$ whose independent sets are the members of $\mathcal{I}$.
\end{theorem}

\begin{definition}
A matroid $M$ is called \textit{connected} if every two elements of it lie in a common circuit or cocircuit.
%More generally, $M$ is called $n$-connected if every $n$ elements of it share a same circuit or cocircuit.
\end{definition}

\begin{definition}
Let $M_1$ and $M_2$ be two matroids with disjoint ground sets. The \textit{direct sum} $M_1 \oplus M_2$ is the matroid whose ground sets is the union of the ground sets of $M_1$ and $M_2$ and whose bases are the union of their bases.
\end{definition}

\begin{theorem}
A matroid $M$ is connected if and only if it cannot be written as a direct sum of two matroids.
\end{theorem}

\begin{definition}
Let $M=(E,\mathcal{B)}$ be a matroid and $X \subset E$.
\begin{enumerate}
    \item The \textit{restriction} of $M$ to $X$, denoted $M|X$ or $M \setminus (E \setminus X)$, is the matroid whose set of bases is the set of the maximal independent sets contained in $X$. In the case of singletons, we will denote $M|{\{e\}}$ and $M\setminus \{e\}$ by $M|{e}$ and $M\setminus e$, respectively. The operation of restricting $M$ to $X$ is called the \textit{deletion} of $E \setminus X$.
    \item The \textit{contraction} of $X$ is the matroid denoted by $M/X$ and given by $(M^* \setminus X)^*$. Similarly to the restriction, we will omit the set brackets when we deal with singletons.
    \item A \textit{minor} of $M$ is a matroid induced from $M$ by a sequence of contractions and restrictions.
\end{enumerate}
\end{definition}

%%%%%%%%%%%%%%%%%%%%%%%%%%%%%%%%%%%%%%%%%%%%%%%%%%%%%%%%%%%%%%
\section{Cluster algebra overview}
This section introduces the notion of cluster algebra. For a wide overview, the reader is referred to \cite{FWZ,FZ}. We begin with the following sequence of definitions.
\begin{definition}
A \textit{(labeled) seed} is a pair $(\textbf{x},B)$ such that $\textbf{x}=(x_1,...,x_n,x_{n+1},...,x_m)$ is a tuple of algebraically independent variables generating a field isomorphic to the field $\mathbb{C}(x_1,...,x_n,x_{n+1},...,x_m)$. Also, $B$ is an $m \times n$ extended skew-symmetrizable matrix, that is, a matrix whose north $n \times n$ submatrix can be transformed to a skew-symmetric matrix by multiplying each row $r_i$ by a nonzero integer $d_i$. The matrix $B$ is called the \textit{exchange matrix} and the tuple $\textbf{x}$ is called the \textit{extended cluster}. The variables $x_1,...,x_n$ are called \textit{mutable}, while the variables $x_{n+1},...,x_m$ are called \textit{frozen}.
\end{definition}

\begin{remark}
In some cases, the field $\mathbb{C}(x_1,...,x_n,x_{n+1},...,x_m)$ of the previous definition is replaced by the field $\mathbb{Q}(x_1,...,x_n,x_{n+1},...,x_m)$. We will mainly deal with the first in this paper, but all the results still make sense for the second.
\end{remark}

\begin{definition}
Let $k$ be an index of a mutable variable of a seed $(\textbf{x},B)$. A \textit{mutation} at $k$ is a transformation to a new seed $(\textbf{x}',B')$ in which $B'$ is an $m \times n$ matrix whose entries are
\begin{equation}
    \label{eq1}
b'_{ij}=\begin {cases}
-b_{ij}, & \text{if}\ i=k \text{ or } j=k,\\ 
b_{ij}+\dfrac{|b_{ik}|b_{kj} + b_{ik}|b_{kj}|}{2}, & \text{otherwise};\\
\end{cases}
\end{equation}
and
$\textbf{x}'=(x_1',...,x_n',x_{n+1}',...,x_m')$ is a tuple such that $x_i'=x_i$ for $i \neq k$ and
$$x_k x_k' = \prod_{b_{ik}>0}x_i^{b_{ik}} + \prod_{b_{ik}<0}x_i^{-b_{ik}}.$$
The seed $(\textbf{x}',B')$ obtained by a mutation at $k$ is denoted sometimes by $\mu_k (\textbf{x},B)$. 
\end{definition}

\begin{remark}
It is not hard to see that the mutation of a seed provides a new seed. Moreover, mutating twice at the same index brings the original seed back. In symbols,
$$\mu_k (\mu_k (\textbf{x},B))=(\textbf{x},B).$$
\end{remark}

\begin{definition}
Let $(\textbf{x},B)$ be a seed. A \textit{cluster algebra (of geometric type)} attached to $(\textbf{x},B)$ is the polynomial algebra $\mathcal{A}=\mathbb{C}[x_{n+1},...,x_{m}][\chi]$, where $\chi$ is the set of all possible mutable variables, that is, the mutable variables of the original seed or a seed obtained by a mutation or a sequence of mutations. The seed $(\textbf{x},B)$ is called the \textit{initial seed}.
\end{definition}

\begin{remark}
By the properties of mutation and algebraically independent sets, it is not hard to see that the cluster algebra attached to some seed is the same cluster algebra attached to any seed mutation. Therefore, in the context of cluster algebras, we often fix an initial seed and describe the cluster algebra $\mathcal{A}$ by means of it. In terms of notation, we sometimes write $\mathcal{A}(\textbf{x},B)$ instead of $\mathcal{A}$, if there is an emphasis on the initial seed $(\textbf{x},B)$.
\end{remark}

\begin{definition}
The \textit{rank} of a seed or a cluster algebra attached to it is the number of mutable variables of its initial seed. A cluster algebra is \textit{of finite type} if it has finitely many seeds. Otherwise, it is \textit{of infinite type}.
\end{definition}

\begin{remark}
The finite type classification of cluster algebras is closely related to Lie Theory. In fact, the cluster algebras of finite type are classified by the Dynkin Diagrams, which are also the main objects classifying the semisimple complex Lie algebras.
\end{remark}

\begin{definition}
Let $B$ be an $n \times n$ square integer matrix. The \emph{Cartan counterpart} of $B$ is the matrix $A(B) = (a_{i,j})$ defined by:
\[
a_{i,i} := 2, \quad \text{and} \quad a_{i,j} := -|b_{i,j}| \ \text{if } i \neq j.
\]
\end{definition}

\begin{theorem}
A cluster algebra $\mathcal{A}$ is of finite type if and only if the Cartan counterpart of one of its seeds is a Cartan matrix of finite type, that is, of type $A_n$, $B_n$, $C_n$, $D_n$, $E_6$, $E_7$, $E_8$, $F_4$, or $G_2$.
\end{theorem}

\begin{definition}
For a cluster algebra $\mathcal{A}$, a \textit{cluster monomial} is a monomial consisting of variables from a single seed.
\end{definition}

%State the classification theorems for type A - type G. Put references FWZ2 and FZ2.
%\section{Finite type classification}
%Just put these as references and remove this section: FWZ2 or FZ2.
%%%%%%%%%%%%%%%%%%%%%%%%%%%%%%%%%%%%%%%%%%%%%%%%%%%%%%%%%%%%%%%%
\section{Cluster matroids}
In this section, we investigate a matroid structure in the set of extended clusters and build a connection between the two topics from there. Throughout, the terms ``cluster algebra" and ``seed" mean a cluster algebra and a seed of finite type.
\begin{example} This example produces a way to give the Grassmannian $\textnormal{Gr}_{2,n}$ a cluster algebra. We skip some details here. For a wider overview, the reader is encouraged to see \cite{FWZ} Section 1.2.  
Consider the octagon of Figure \ref{fig 1}. A \textit{triangulation} of the octagon is a shape obtained by drawing a maximal number of pairwise non-crossing diagonals. It is easily seen that any triangulation of the octagon produces exactly 5 non-crossing diagonals, one of them is the one in Figure \ref{fig 1}. More generally, a triangulation of an $m$-gon produces exactly $m-3$ non-crossing diagonals. Now, this octagon forms a combinatorial way of describing a seed whose frozen variables are the sides of the octagon and whose mutable variables are the non-crossing diagonals. The mutation of seeds here corresponds to diagonal flipping. For instance, flipping the diagonal $P_{58}$ to $P_{16}$ corresponds to another triangulation that is a mutation of the first one at the variable $P_{58}$. Another example is to flip $P_{68}$ to $P_{57}$ and so on.
\begin{figure}
    \centering
    \begin{tikzpicture}
    [dot/.style={circle,fill,inner sep=0pt,minimum size=4pt}]
 \draw node[regular polygon,regular polygon sides=8,draw,minimum size=4cm]
 (p8){}
  foreach \X[count=\Y] in {1,2,3,8} {(p8.corner \X) node[dot](\X){} node[above]{\X}}
  foreach \X[count=\Y] in {4,5,6,7} {(p8.corner \X) node[dot](\X){} node[below]{\X}}
%    -- (p8.corner \Y) node[dot](n-\Y){} }
  foreach \X[count=\Y] in {5} {(p8.corner \X) node[dot](n-\X){}
    -- (p8.corner \Y) node[dot](n-\Y){} node[midway,sloped,above] {$P_{15}$}}
    {(p8.corner 3) node[dot](n-3){}
    -- (p8.corner 1) node[dot](n-1){} 
    node[midway,sloped,below] {$P_{13}$}}
    {(p8.corner 3) node[dot](n-3){}
    -- (p8.corner 5) node[dot](n-5){} 
    node[midway,sloped,above] {$P_{35}$}}
    {(p8.corner 6) node[dot](n-6){}
    -- (p8.corner 8) node[dot](n-8){} node[midway,sloped,above] {$P_{68}$}}
    {(p8.corner 5) node[dot](n-5){}
    -- (p8.corner 8) node[dot](n-8){} node[midway,sloped,above] {$P_{58}$}}
    ;
\end{tikzpicture}
%}
    \caption{A triangulation of the octagon.}
    \label{fig 1}
\end{figure}
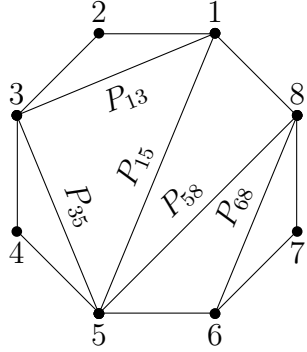

Na\"ively, one might think that this gives rise to a matroid on the set of diagonals of an $n$-gon by taking as the bases those edges that are the edges of a triangulation, but this turns out to be false. In fact, consider the triangulation in Figure \ref{fig 2}. If $P_{14}$ is removed from this triangulation, then there is no diagonal from the triangulation of Figure \ref{fig 1} that can be inserted and give a new triangulation. Hence, the second axiom of matroid bases is not satisfied and this is not a matroid. However, we will be able to solve this issue in the next remark.

%However, as polygons and their triangulations are closely related to a cluster algebra structure of the Grassmannian $\textnormal{Gr}_{2,n}$, it would be interesting if this example can be refined to form a new class of matroids coming from polygons and their triangulations.
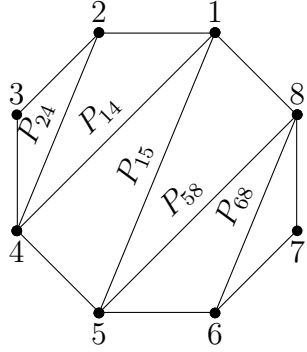
\begin{figure}
    \centering
\begin{tikzpicture}
    [dot/.style={circle,fill,inner sep=0pt,minimum size=4pt}]
 \draw node[regular polygon,regular polygon sides=8,draw,minimum size=4cm]
 (p8){}
  foreach \X[count=\Y] in {1,2,3,8} {(p8.corner \X) node[dot](\X){} node[above]{\X}}
  foreach \X[count=\Y] in {4,5,6,7} {(p8.corner \X) node[dot](\X){} node[below]{\X}}
%    -- (p8.corner \Y) node[dot](n-\Y){} }
  foreach \X[count=\Y] in {5} {(p8.corner \X) node[dot](n-\X){}
    -- (p8.corner \Y) node[dot](n-\Y){} node[midway,sloped,above] {$P_{15}$}}
    {(p8.corner 4) node[dot](n-3){}
    -- (p8.corner 1) node[dot](n-1){} 
    node[midway,sloped,above] {$P_{14}$}}
    {(p8.corner 2) node[dot](n-1){}
    -- (p8.corner 4) node[dot](n-4){} 
    node[midway,sloped,above] {$P_{24}$}}
    {(p8.corner 6) node[dot](n-6){}
    -- (p8.corner 8) node[dot](n-8){} node[midway,sloped,above] {$P_{68}$}}
    {(p8.corner 5) node[dot](n-5){}
    -- (p8.corner 8) node[dot](n-8){} node[midway,sloped,above] {$P_{58}$}}
    ;
\end{tikzpicture}
\caption{Another triangulation of the octagon.}
\label{fig 2}
\end{figure}
\end{example}
\begin{remark}
From Theorem \ref{matroid structure}, it is straightforward to see that the set of extended clusters of some cluster algebra induces a matroid. Indeed, define the bases of this matroid to be all maximal algebraically independent sets containing the frozen variables, where the ground set is the set of all possible variables generated by any sequence of mutations. We will call this a \textit{cluster matroid}.
%it is straightforward to see that the set of extended clusters of some cluster algebra forms a matroid. In fact, since each extended cluster forms an algebraically independent set, the collection of all possible extended clusters forms the bases of a matroid whose ground set $E$ is the set of all possible mutable and frozen variables. In this paper, this matroid will be called a \textit{cluster matroid}. 
\end{remark}

\begin{remark}
Since cluster algebras of type $A_n$ are combinatorially described by the triangulations of an $(n+3)$-gon and since the $(2n+3)$-quantities attached to the triangulations are algebraically independent, one can see that this can give a matroid on any $n$-gon using the algebraic independence and the definition of cluster matroids of the previous remark. 
\end{remark}
%Let us now characterize the cluster matroids induced by the small rank cluster algebras:

\begin{example}[Rank 1] (c.f. Example 3.2.2 of \cite{FWZ})
A cluster algebra of rank 1 has exactly two mutable variables, say $x_1$ and $x'_1$, each of them is a mutation of the other. It can have any number of frozen variables. Any $m \times 1$ matrix $B$ with top entry 0 is a possible extended exchange matrix for such a cluster algebra. Of course, $x_1$ and $x_1'$ are related by the mutation relation $x_1x_1'= \prod_i x_i^{b_i} + \prod_j x_j^{-b_j}$ where $i \neq j$ and $b_i > 0$ and $b_j <0$ for all $i$ and $j$. This cluster algebra is generated by the variables $x_1,x_1',x_2,...,x_m$ and lives inside $\CC(x_1,x_2,...,x_m)=\CC(x_1',x_2,...,x_m)$. The matroid attached to this cluster algebra is given by the ground set $E=\{x_1,x_1',x_2,...,x_m\}$ and the bases $B_1=\{x_1,x_2,...,x_m \}$ and $B_2=\{x_1',x_2,...,x_m \}$. A more concrete example is the coordinate ring of the subgroup of unipotent upper triangular matrices
$$U^+= \left\{ \begin{bmatrix}
1 & a & b\\
0 & 1 & c\\
0 & 0 & 1
\end{bmatrix}
\right\} \subset SL_3.$$
This coordinate ring is $\CC[a,b,c]$ and it forms a cluster algebra of rank 1 whose extended clusters are $\{a,b,ac-b\}$ and $\{c,b,ac-b\}$. Clearly, the mutable variables are $a$ and $c$ and the frozen variables are $b$ and $ac-b$.
 \end{example}
 \begin{example} \label{uniform}
 Let $(x_1,x_2)$ be a seed in which both of its variables are mutable. Let
 $$B=\begin{bmatrix}
 0 & 1\\
 -1 & 0
\end{bmatrix}$$
be the exchange matrix attached to this seed. It is not hard to see that the list of all possible cluster variables is: $x_1$, $x_2$, $\dfrac{1+x_2}{x_1}$, $\dfrac{1+x_1+x_2}{x_1x_2}$ and $\dfrac{1+x_1}{x_2}$. 
%The cluster matroid induced by this list has the following set of bases:
%$$\mathcal{B}= \bigg\{ \{x_1,x_2\}, \{x_1,\dfrac{1+x_2}{x_1}\}, \{x_2,\dfrac{1+x_1}{x_2} \}, \{x_1, \dfrac{1+x_1+x_2}{x_1x_2} \}, \{x_2, \dfrac{1+x_1+x_2}{x_1x_2} \} \bigg \}$$

Any 2-set of elements of this list is a basis of the induced cluster matroid. This matroid is denoted by $U_{2,5}$. More generally, for $n \leq m$, the \textit{uniform matroid} $U_{n,m}$ is the matroid whose ground set is $\{1,...,m \}$ such that any $n$-subset forms a basis.
 \end{example}
 
 \begin{remark}
In general, cluster matroids are not closed under duality. This is because the number of seeds is fully determined by the exchange matrix together with the cluster algebra rank. Thus, if the rank is changed, then the number of seeds, and hence the size of the matroid, will be different.
 \end{remark}

 \begin{remark}
The matroids introduced in Theorem \ref{matroid structure} are called \textit{algebraic}. There are many open questions about them and their duals. We believe that the study of cluster matroids can help answer some of these open questions. One can see \cite{O} for a deeper look at algebraic matroids.
 \end{remark}

Now, we note the following results of cluster matroids:

\begin{theorem}
A cluster matroid of size greater than or equal to 2 is connected if and only if it has no frozen variables.
\end{theorem}

\begin{proof}
Since the frozen variables appear in every basis, they are coloops, that is, codependent singletons. It is straight forward to verify that any matroid with (co)loops is disconnected.
%If a cluster matroid has no mutable variables, then it has no circuits. It is not difficult to see that this implies disconnectedness if the size is greater than or equal to 2.
%On the other hand, let $(x_1,...,x_n,x_{n+1},...,x_m)$ be an extended cluster whose first $n$ variables are mutable and the others are frozen and let $B$ be an exchange matrix. Then $M=M_1 \oplus M_2$, where $M_1$ is the matroid whose bases are the seeds after the deletion of the frozen variables and $M_2$ is the matroid whose single basis is $\{x_{n+1},...,x_m \}$.
Now, assume that there is a disconnected cluster matroid $M$ consisting merely of mutable variables. This implies the existence of two matroids $M_1$ and $M_2$ such that $M=M_1 \oplus M_2$. Assume that $(x_1,...,x_n,...,x_m)$ is an extended cluster and suppose that $x_i \in M_1$. If $x_i$ is mutable, then its mutation $x'_i \in M_1$. Otherwise, $M$ has two bases $\{ x_1,...,x_i,...,x_n \}$ and $\{ x_1,...,x'_i,...,x_n \}$ such that the first has elements from $M_1$ more than the second, a contradiction. Similarly, if $x_j$ is a mutable variable living in $M_2$, then its mutation $x'_j \in M_2$. Now, assume without loss of generality that there exists a number $r$ such that $x_1,...,x_r \in M_1$ and $x_{r+1},...,x_n \in M_2$. Then, at the first level, any mutation at mutable indices of the first $r$ spots induces a variable in $M_1$. Likewise, any mutation at the rest spots induces a variable in $M_2$. At the second level, a second mutation at $k \in [1,r]$ produces a variable in $M_1$, no matter if the first mutation was at $[1,r]$ or $[r+1,m]$. Similarly, a second mutation at $k \in [r+1,m]$ induces a variable in $M_2$. This continues to any level of mutations. Note that $x_1$ and $x_{r+1}$ must not live in a same circuit; otherwise, there is a connected component containing both of them, which means that they must be both in $M_1$ or both in $M_2$. However, $\{x_1,x_1',x_2,...,x_r,,x_{r+1},...,x_n\}$ is dependent and has no dependent subset. Hence, it is a circuit, a contradiction. Since the mutation at an index produces a variable in the same original connected component, the previous argument can be generalized for any two variables by comparing their mutations at some certain level.
\end{proof}

In Example \ref{uniform}, we have seen a connection between cluster matroids and uniform matroids. This guides us to the following theorem:

\begin{theorem}
Let $M$ be a cluster matroid whose initial seed has no frozen variables. Then $M$ is equal to $U_{n,|M|}$, where $n$ is the size of the initial seed and $|M|$ is the ground set cardinality.
\end{theorem}
\begin{proof}
First, note that any $n$-subset of the ground set is algebraically independent. It is not hard to see this, since the initial seed contains $n$ algebraically independent variables and the mutation formula
$$x_k x_k' = \prod_{b_{ik}>0}x_i^{b_{ik}} + \prod_{b_{ik}<0}x_i^{-b_{ik}}$$
implies that the variable $x_k'$ preserves the apperance of all the variables at the initial seed, no more no less.

Second, any $(n+1)$-subset of the ground set is algebraically dependent. Indeed, since all cluster variables lie in the rational function field $\mathbb{C}(x_1, \ldots, x_n)$, whose transcendence degree over $\mathbb{C}$ is $n$, no subset of mutable variables of size greater than $n$ can be algebraically independent. Thus, any $(n+1)$-subset of cluster variables must be algebraically dependent.

Clearly, the combination of the results of the two paragraphs above implies that $M=U_{n,|M|}$.
\end{proof}

We have seen that cluster algebras induce matroids via the property of algebraic independence. In fact, we can get another type of matroid from cluster algebras. This comes from the \textit{Laurent phenomenon}, which is one of the most powerful phenomena of cluster algebras:

\begin{theorem}[Laurent phenomenon]
Let $\mathcal{A}$ be a cluster algebra. Any cluster variable of $\mathcal{A}$ can be expressed as a Laurent polynomial in the variables of any extended cluster with integer coefficients. Moreover, the frozen variables do not appear in the denominator of any such Laurent polynomial.
\end{theorem}

\begin{corollary}
The cluster monomials of any cluster algebra are linearly independent over the ground field $(\QQ$ or $\CC)$.
\end{corollary}

\begin{corollary}
Cluster monomials of a cluster algebra of finite type form a representable matroid.
\end{corollary}
%\section{duals}
%In this section, we study the duals of cluster matroids.
\section{minors}
To better understand the class of cluster matroids, it is important to see how it behaves after contraction. Since the cluster algebras of finite type are classified by Dynkin diagrams (or Cartan matrices), we choose one example of type $A_2$ to buid some results from. Clearly, the mutable and frozen variables are the main ingredients to determine the seed pattern of a cluster algebra. Thus, we look closely to the effect of each of them.\\

We first make the following note regarding the contraction by frozen variables in a cluster matroid.
\begin{theorem}
The contraction of a frozen variable from a cluster matroid is a cluster matroid.
\end{theorem}

\begin{proof}
Assume that $M$ is a cluster matroid and $e$ is a frozen variable. We have the following:
    \begin{align*}
    \mathcal{B}(M/e) &=\Big \{ X \subset E-e : M|_e \textnormal{ has a basis } B \textnormal{ such that } X \cup B \in \mathcal{B}(M) \Big \}\\
    &=\Big \{ X \subset E-e : X \cup e \in \mathcal{B}(M) \Big \}\\
    &=\Big \{ B-e : B \in \mathcal{B}(M) \Big \};
    \end{align*}
    where the last equality holds because of the facts that $e$ is frozen and frozen variables appear in every cluster-matroid basis. Note that this is the same cluster matroid obtained by removing the frozen variable $e$ from its original seed together with its corresponding row from the exchange matrix $B$.\\
\end{proof}

\begin{corollary}
If $M$ is a cluster matroid and $X$ is a set of frozen variables, then $M/X$ is a cluster matroid.
\end{corollary}

\begin{remark}
Based on the definition of a cluster matroid, it is straightforward to see that the deletion of a frozen variable results in the same as the contraction of it. Therefore, the deletion of a frozen variable from a cluster matroid is a cluster matroid as well.
\end{remark}

Let us recall this theorem for the classification of type $A_n$:

\begin{theorem}
Let $\mathcal A$ be a cluster algebra of type $A_n$. Then, its clusters can be labeled by the diagonals of a convex $(n+3)$-gon so that
\begin{itemize}
    \item clusters correspond to triangulations of the gon $P_{n+3}$ by noncrossing diagonals,
    \item mutations correspond to flips, and
    \item exchange matrices are given such that their counterparts are Cartan matrices of type $A_n$.
\end{itemize}
Cluster variables labeled by different diagonals are distinct, so there are altogether $\frac{n(n+3)}{2}$ cluster variables and $\frac{1}{n+2} \binom{2n+2}{n+1}$ seeds.
\end{theorem}

We use this theorem now to develop the following example:

\begin{example}
Let us focus on a cluster matroid of type $A_2$. The set of extended clusters can be discribed combinatorially using the following graph, where the diagonals represent mutable variables and the sides represent the frozen ones:\\

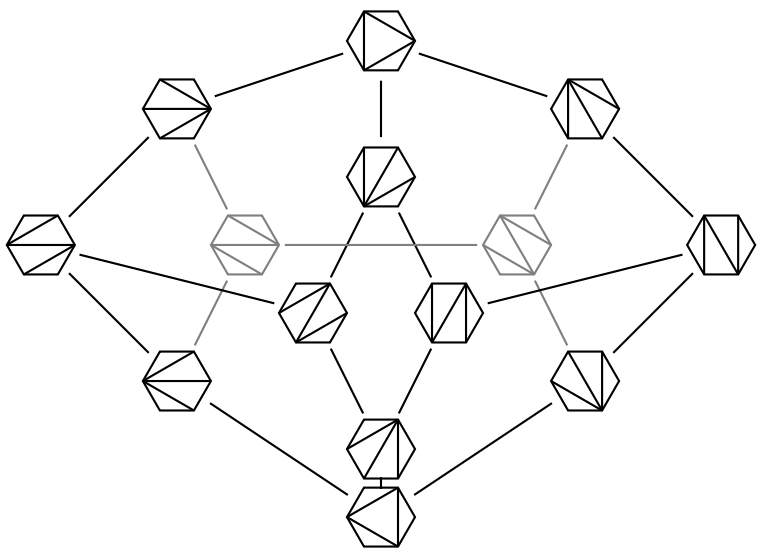
\begin{center}
\begin{tikzpicture}[scale=0.9]

% Hexagon 1 at (0,-1)
\coordinate (A1) at (0,-1);
\path (A1) ++(0:0.5) coordinate (A1A);
\path (A1) ++(60:0.5) coordinate (A1B);
\path (A1) ++(120:0.5) coordinate (A1C);
\path (A1) ++(180:0.5) coordinate (A1D);
\path (A1) ++(240:0.5) coordinate (A1E);
\path (A1) ++(300:0.5) coordinate (A1F);
\draw[thick] (A1A) -- (A1B) -- (A1C) -- (A1D) -- (A1E) -- (A1F) -- cycle;
\draw[thick] (A1B) -- (A1F);
\draw[thick] (A1F) -- (A1D);
\draw[thick] (A1D) -- (A1B);

% Hexagon 2 at (0,0)
\coordinate (A2) at (0,0);
\path (A2) ++(0:0.5) coordinate (A2A);
\path (A2) ++(60:0.5) coordinate (A2B);
\path (A2) ++(120:0.5) coordinate (A2C);
\path (A2) ++(180:0.5) coordinate (A2D);
\path (A2) ++(240:0.5) coordinate (A2E);
\path (A2) ++(300:0.5) coordinate (A2F);
\draw[thick] (A2A) -- (A2B) -- (A2C) -- (A2D) -- (A2E) -- (A2F) -- cycle;
\draw[thick] (A2B) -- (A2D);
\draw[thick] (A2B) -- (A2F);
\draw[thick] (A2B) -- (A2E);

% Hexagon 3 at (3,1)
\coordinate (A3) at (3,1);
\path (A3) ++(0:0.5) coordinate (A3A);
\path (A3) ++(60:0.5) coordinate (A3B);
\path (A3) ++(120:0.5) coordinate (A3C);
\path (A3) ++(180:0.5) coordinate (A3D);
\path (A3) ++(240:0.5) coordinate (A3E);
\path (A3) ++(300:0.5) coordinate (A3F);
\draw[thick] (A3A) -- (A3B) -- (A3C) -- (A3D) -- (A3E) -- (A3F) -- cycle;
\draw[thick] (A3B) -- (A3F);
\draw[thick] (A3C) -- (A3F);
\draw[thick] (A3D) -- (A3F);

% Hexagon 4 at (-3,1)
\coordinate (A4) at (-3,1);
\path (A4) ++(0:0.5) coordinate (A4A);
\path (A4) ++(60:0.5) coordinate (A4B);
\path (A4) ++(120:0.5) coordinate (A4C);
\path (A4) ++(180:0.5) coordinate (A4D);
\path (A4) ++(240:0.5) coordinate (A4E);
\path (A4) ++(300:0.5) coordinate (A4F);
\draw[thick] (A4A) -- (A4B) -- (A4C) -- (A4D) -- (A4E) -- (A4F) -- cycle;
\draw[thick] (A4A) -- (A4D);
\draw[thick] (A4B) -- (A4D);
\draw[thick] (A4F) -- (A4D);

% Hexagon 5 at (-1,2)
\coordinate (A5) at (-1,2);
\path (A5) ++(0:0.5) coordinate (A5A);
\path (A5) ++(60:0.5) coordinate (A5B);
\path (A5) ++(120:0.5) coordinate (A5C);
\path (A5) ++(180:0.5) coordinate (A5D);
\path (A5) ++(240:0.5) coordinate (A5E);
\path (A5) ++(300:0.5) coordinate (A5F);
\draw[thick] (A5A) -- (A5B) -- (A5C) -- (A5D) -- (A5E) -- (A5F) -- cycle;
\draw[thick] (A5B) -- (A5D);
\draw[thick] (A5B) -- (A5E);
\draw[thick] (A5A) -- (A5E);

% Hexagon 6 at (1,2)
\coordinate (A6) at (1,2);
\path (A6) ++(0:0.5) coordinate (A6A);
\path (A6) ++(60:0.5) coordinate (A6B);
\path (A6) ++(120:0.5) coordinate (A6C);
\path (A6) ++(180:0.5) coordinate (A6D);
\path (A6) ++(240:0.5) coordinate (A6E);
\path (A6) ++(300:0.5) coordinate (A6F);
\draw[thick] (A6A) -- (A6B) -- (A6C) -- (A6D) -- (A6E) -- (A6F) -- cycle;
\draw[thick] (A6B) -- (A6F);
\draw[thick] (A6B) -- (A6E);
\draw[thick] (A6C) -- (A6E);

% Hexagon 7 at (-2,4)
%%%%%%%%%%%%%%%%%%%%%%%%%%%%%%%%%%%%%%%%%%%%%%%%%%%%%%%%%%SKIPPED
\coordinate (A7) at (-2,3);
\path (A7) ++(0:0.5) coordinate (A7A);
\path (A7) ++(60:0.5) coordinate (A7B);
\path (A7) ++(120:0.5) coordinate (A7C);
\path (A7) ++(180:0.5) coordinate (A7D);
\path (A7) ++(240:0.5) coordinate (A7E);
\path (A7) ++(300:0.5) coordinate (A7F);
\draw[thick, gray] (A7A) -- (A7B) -- (A7C) -- (A7D) -- (A7E) -- (A7F) -- cycle;
\draw[thick, gray] (A7A) -- (A7C);
\draw[thick, gray] (A7A) -- (A7D);
\draw[thick, gray] (A7D) -- (A7F);

% Hexagon 8 at (-5,3)
\coordinate (A8) at (-5,3);
\path (A8) ++(0:0.5) coordinate (A8A);
\path (A8) ++(60:0.5) coordinate (A8B);
\path (A8) ++(120:0.5) coordinate (A8C);
\path (A8) ++(180:0.5) coordinate (A8D);
\path (A8) ++(240:0.5) coordinate (A8E);
\path (A8) ++(300:0.5) coordinate (A8F);
\draw[thick] (A8A) -- (A8B) -- (A8C) -- (A8D) -- (A8E) -- (A8F) -- cycle;
\draw[thick] (A8A) -- (A8E);
\draw[thick] (A8A) -- (A8D);
\draw[thick] (A8B) -- (A8D);

% Hexagon 9 at (2,3)
\coordinate (A9) at (2,3);
\path (A9) ++(0:0.5) coordinate (A9A);
\path (A9) ++(60:0.5) coordinate (A9B);
\path (A9) ++(120:0.5) coordinate (A9C);
\path (A9) ++(180:0.5) coordinate (A9D);
\path (A9) ++(240:0.5) coordinate (A9E);
\path (A9) ++(300:0.5) coordinate (A9F);
\draw[thick, gray] (A9A) -- (A9B) -- (A9C) -- (A9D) -- (A9E) -- (A9F) -- cycle;
\draw[thick, gray] (A9A) -- (A9C);
\draw[thick, gray] (A9C) -- (A9F);
\draw[thick, gray] (A9F) -- (A9D);

% Hexagon 10 at (5,3)
\coordinate (A10) at (5,3);
\path (A10) ++(0:0.5) coordinate (A10A);
\path (A10) ++(60:0.5) coordinate (A10B);
\path (A10) ++(120:0.5) coordinate (A10C);
\path (A10) ++(180:0.5) coordinate (A10D);
\path (A10) ++(240:0.5) coordinate (A10E);
\path (A10) ++(300:0.5) coordinate (A10F);
\draw[thick] (A10A) -- (A10B) -- (A10C) -- (A10D) -- (A10E) -- (A10F) -- cycle;
\draw[thick] (A10B) -- (A10F);
\draw[thick] (A10F) -- (A10C);
\draw[thick] (A10C) -- (A10E);

% Hexagon 11 at (0,4)
%%%%%%%%%%%%%%%%%%%%%%%%%%%%%%%%%%%SKIPPED
\coordinate (A11) at (0,4);
\path (A11) ++(0:0.5) coordinate (A11A);
\path (A11) ++(60:0.5) coordinate (A11B);
\path (A11) ++(120:0.5) coordinate (A11C);
\path (A11) ++(180:0.5) coordinate (A11D);
\path (A11) ++(240:0.5) coordinate (A11E);
\path (A11) ++(300:0.5) coordinate (A11F);
\draw[thick] (A11A) -- (A11B) -- (A11C) -- (A11D) -- (A11E) -- (A11F) -- cycle;
\draw[thick] (A11A) -- (A11E);
\draw[thick] (A11B) -- (A11E);
\draw[thick] (A11C) -- (A11E);

% Hexagon 12 at (-3,5)
\coordinate (A12) at (-3,5);
\path (A12) ++(0:0.5) coordinate (A12A);
\path (A12) ++(60:0.5) coordinate (A12B);
\path (A12) ++(120:0.5) coordinate (A12C);
\path (A12) ++(180:0.5) coordinate (A12D);
\path (A12) ++(240:0.5) coordinate (A12E);
\path (A12) ++(300:0.5) coordinate (A12F);
\draw[thick] (A12A) -- (A12B) -- (A12C) -- (A12D) -- (A12E) -- (A12F) -- cycle;
\draw[thick] (A12A) -- (A12C);
\draw[thick] (A12A) -- (A12D);
\draw[thick] (A12A) -- (A12E);

% Hexagon 13 at (3,5)
\coordinate (A13) at (3,5);
\path (A13) ++(0:0.5) coordinate (A13A);
\path (A13) ++(60:0.5) coordinate (A13B);
\path (A13) ++(120:0.5) coordinate (A13C);
\path (A13) ++(180:0.5) coordinate (A13D);
\path (A13) ++(240:0.5) coordinate (A13E);
\path (A13) ++(300:0.5) coordinate (A13F);
\draw[thick] (A13A) -- (A13B) -- (A13C) -- (A13D) -- (A13E) -- (A13F) -- cycle;
\draw[thick] (A13A) -- (A13C);
\draw[thick] (A13C) -- (A13E);
\draw[thick] (A13C) -- (A13F);

% Hexagon 14 at (0,6)
\coordinate (A14) at (0,6);
\path (A14) ++(0:0.5) coordinate (A14A);
\path (A14) ++(60:0.5) coordinate (A14B);
\path (A14) ++(120:0.5) coordinate (A14C);
\path (A14) ++(180:0.5) coordinate (A14D);
\path (A14) ++(240:0.5) coordinate (A14E);
\path (A14) ++(300:0.5) coordinate (A14F);
\draw[thick] (A14A) -- (A14B) -- (A14C) -- (A14D) -- (A14E) -- (A14F) -- cycle;
\draw[thick] (A14A) -- (A14C);
\draw[thick] (A14C) -- (A14E);
\draw[thick] (A14E) -- (A14A);

% Edges between the triangulations (centers)

\draw[thick, shorten <=15pt, shorten >=15pt] (A1) -- (A2);
\draw[thick, shorten <=15pt, shorten >=15pt] (A1) -- (A3);
\draw[thick, shorten <=15pt, shorten >=15pt] (A1) -- (A4);
\draw[thick, shorten <=15pt, shorten >=15pt] (A2) -- (A5);
\draw[thick, shorten <=15pt, shorten >=15pt] (A2) -- (A6);
\draw[thick, gray, shorten <=15pt, shorten >=15pt] (A3) -- (A9);
\draw[thick, shorten <=15pt, shorten >=15pt] (A3) -- (A10);
\draw[thick, gray, shorten <=15pt, shorten >=15pt] (A4) -- (A7);
\draw[thick, shorten <=15pt, shorten >=15pt] (A4) -- (A8);
\draw[thick, shorten <=15pt, shorten >=15pt] (A5) -- (A8);
\draw[thick, shorten <=15pt, shorten >=15pt] (A5) -- (A11);
\draw[thick, shorten <=15pt, shorten >=15pt] (A6) -- (A10);
\draw[thick, shorten <=15pt, shorten >=15pt] (A6) -- (A11);
\draw[thick, gray, shorten <=15pt, shorten >=15pt] (A7) -- (A9);
\draw[thick, gray, shorten <=15pt, shorten >=15pt] (A7) -- (A12);
\draw[thick, shorten <=15pt, shorten >=15pt] (A8) -- (A12);
\draw[thick, gray, shorten <=15pt, shorten >=15pt] (A9) -- (A13);
\draw[thick, shorten <=15pt, shorten >=15pt] (A10) -- (A13);
\draw[thick, shorten <=15pt, shorten >=15pt] (A11) -- (A14);
\draw[thick, shorten <=15pt, shorten >=15pt] (A12) -- (A14);
\draw[thick, shorten <=15pt, shorten >=15pt] (A13) -- (A14);

\end{tikzpicture}
\end{center}

It is easily seen that the contraction (or deletion) of a frozen variable, for example the top side of each hexagon, will produce bases that can be described using this graph, where again the sides are frozen and diagonals are mutable:

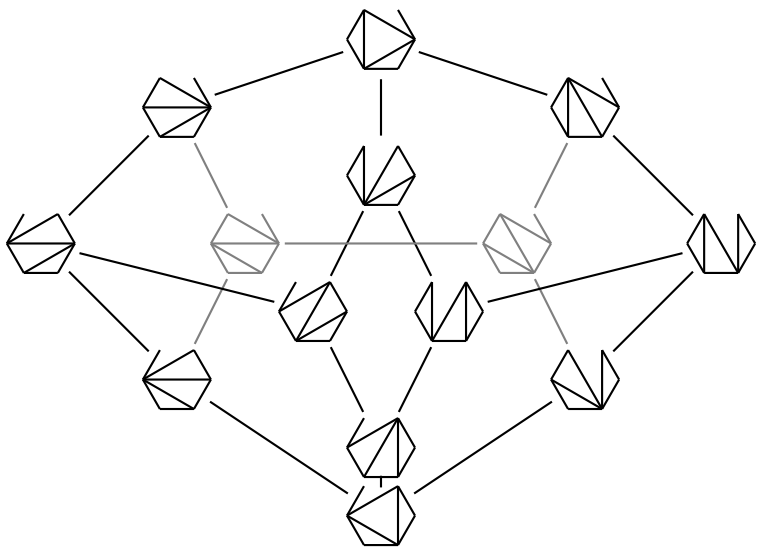
\begin{center}
\begin{tikzpicture}[scale=0.9]

% Hexagon 1 at (0,-1)
\coordinate (A1) at (0,-1);
\path (A1) ++(0:0.5) coordinate (A1A);
\path (A1) ++(60:0.5) coordinate (A1B);
\path (A1) ++(120:0.5) coordinate (A1C);
\path (A1) ++(180:0.5) coordinate (A1D);
\path (A1) ++(240:0.5) coordinate (A1E);
\path (A1) ++(300:0.5) coordinate (A1F);

\draw[thick] (A1A) -- (A1B);
%\draw[thick] (A1B) -- (A1C);
\draw[thick] (A1C) -- (A1D);
\draw[thick] (A1D) -- (A1E);
\draw[thick] (A1E) -- (A1F);
\draw[thick] (A1F) -- (A1A);

\draw[thick] (A1B) -- (A1F);
\draw[thick] (A1F) -- (A1D);
\draw[thick] (A1D) -- (A1B);

% Hexagon 2 at (0,0)
\coordinate (A2) at (0,0);
\path (A2) ++(0:0.5) coordinate (A2A);
\path (A2) ++(60:0.5) coordinate (A2B);
\path (A2) ++(120:0.5) coordinate (A2C);
\path (A2) ++(180:0.5) coordinate (A2D);
\path (A2) ++(240:0.5) coordinate (A2E);
\path (A2) ++(300:0.5) coordinate (A2F);

\draw[thick] (A2A) -- (A2B);
%\draw[thick] (A2B) -- (A2C);
\draw[thick] (A2C) -- (A2D);
\draw[thick] (A2D) -- (A2E);
\draw[thick] (A2E) -- (A2F);
\draw[thick] (A2F) -- (A2A);

\draw[thick] (A2B) -- (A2D);
\draw[thick] (A2B) -- (A2F);
\draw[thick] (A2B) -- (A2E);

% Hexagon 3 at (3,1)
\coordinate (A3) at (3,1);
\path (A3) ++(0:0.5) coordinate (A3A);
\path (A3) ++(60:0.5) coordinate (A3B);
\path (A3) ++(120:0.5) coordinate (A3C);
\path (A3) ++(180:0.5) coordinate (A3D);
\path (A3) ++(240:0.5) coordinate (A3E);
\path (A3) ++(300:0.5) coordinate (A3F);
\draw[thick] (A3A) -- (A3B);
%\draw[thick] (A3B) -- (A3C);
\draw[thick] (A3C) -- (A3D);
\draw[thick] (A3D) -- (A3E);
\draw[thick] (A3E) -- (A3F);
\draw[thick] (A3F) -- (A3A);
\draw[thick] (A3B) -- (A3F);
\draw[thick] (A3C) -- (A3F);
\draw[thick] (A3D) -- (A3F);

% Hexagon 4 at (-3,1)
\coordinate (A4) at (-3,1);
\path (A4) ++(0:0.5) coordinate (A4A);
\path (A4) ++(60:0.5) coordinate (A4B);
\path (A4) ++(120:0.5) coordinate (A4C);
\path (A4) ++(180:0.5) coordinate (A4D);
\path (A4) ++(240:0.5) coordinate (A4E);
\path (A4) ++(300:0.5) coordinate (A4F);
\draw[thick] (A4A) -- (A4B);
%\draw[thick] (A4B) -- (A4C);
\draw[thick] (A4C) -- (A4D);
\draw[thick] (A4D) -- (A4E);
\draw[thick] (A4E) -- (A4F);
\draw[thick] (A4F) -- (A4A);
\draw[thick] (A4A) -- (A4D);
\draw[thick] (A4B) -- (A4D);
\draw[thick] (A4F) -- (A4D);

% Hexagon 5 at (-1,2)
\coordinate (A5) at (-1,2);
\path (A5) ++(0:0.5) coordinate (A5A);
\path (A5) ++(60:0.5) coordinate (A5B);
\path (A5) ++(120:0.5) coordinate (A5C);
\path (A5) ++(180:0.5) coordinate (A5D);
\path (A5) ++(240:0.5) coordinate (A5E);
\path (A5) ++(300:0.5) coordinate (A5F);
\draw[thick] (A5A) -- (A5B);
%\draw[thick] (A5B) -- (A5C);
\draw[thick] (A5C) -- (A5D);
\draw[thick] (A5D) -- (A5E);
\draw[thick] (A5E) -- (A5F);
\draw[thick] (A5F) -- (A5A);
\draw[thick] (A5B) -- (A5D);
\draw[thick] (A5B) -- (A5E);
\draw[thick] (A5A) -- (A5E);

% Hexagon 6 at (1,2)
\coordinate (A6) at (1,2);
\path (A6) ++(0:0.5) coordinate (A6A);
\path (A6) ++(60:0.5) coordinate (A6B);
\path (A6) ++(120:0.5) coordinate (A6C);
\path (A6) ++(180:0.5) coordinate (A6D);
\path (A6) ++(240:0.5) coordinate (A6E);
\path (A6) ++(300:0.5) coordinate (A6F);
\draw[thick] (A6A) -- (A6B);
%\draw[thick] (A6B) -- (A6C);
\draw[thick] (A6C) -- (A6D);
\draw[thick] (A6D) -- (A6E);
\draw[thick] (A6E) -- (A6F);
\draw[thick] (A6F) -- (A6A);
\draw[thick] (A6B) -- (A6F);
\draw[thick] (A6B) -- (A6E);
\draw[thick] (A6C) -- (A6E);

% Hexagon 7 at (-2,4)
%%%%%%%%%%%%%%%%%%%%%%%%%%%%%%%%%%%%%%%%%%%%%%%%%%%%%%%%%%SKIPPED
\coordinate (A7) at (-2,3);
\path (A7) ++(0:0.5) coordinate (A7A);
\path (A7) ++(60:0.5) coordinate (A7B);
\path (A7) ++(120:0.5) coordinate (A7C);
\path (A7) ++(180:0.5) coordinate (A7D);
\path (A7) ++(240:0.5) coordinate (A7E);
\path (A7) ++(300:0.5) coordinate (A7F);
\draw[thick, gray] (A7A) -- (A7B);
%\draw[thick] (A7B) -- (A7C);
\draw[thick, gray] (A7C) -- (A7D);
\draw[thick, gray] (A7D) -- (A7E);
\draw[thick, gray] (A7E) -- (A7F);
\draw[thick, gray] (A7F) -- (A7A);
\draw[thick, gray] (A7A) -- (A7C);
\draw[thick, gray] (A7A) -- (A7D);
\draw[thick, gray] (A7D) -- (A7F);

% Hexagon 8 at (-5,3)
\coordinate (A8) at (-5,3);
\path (A8) ++(0:0.5) coordinate (A8A);
\path (A8) ++(60:0.5) coordinate (A8B);
\path (A8) ++(120:0.5) coordinate (A8C);
\path (A8) ++(180:0.5) coordinate (A8D);
\path (A8) ++(240:0.5) coordinate (A8E);
\path (A8) ++(300:0.5) coordinate (A8F);
\draw[thick] (A8A) -- (A8B);
%\draw[thick] (A8B) -- (A8C);
\draw[thick] (A8C) -- (A8D);
\draw[thick] (A8D) -- (A8E);
\draw[thick] (A8E) -- (A8F);
\draw[thick] (A8F) -- (A8A);
\draw[thick] (A8A) -- (A8E);
\draw[thick] (A8A) -- (A8D);
\draw[thick] (A8B) -- (A8D);

% Hexagon 9 at (2,3)
\coordinate (A9) at (2,3);
\path (A9) ++(0:0.5) coordinate (A9A);
\path (A9) ++(60:0.5) coordinate (A9B);
\path (A9) ++(120:0.5) coordinate (A9C);
\path (A9) ++(180:0.5) coordinate (A9D);
\path (A9) ++(240:0.5) coordinate (A9E);
\path (A9) ++(300:0.5) coordinate (A9F);
\draw[thick, gray] (A9A) -- (A9B);
%\draw[thick] (A9B) -- (A9C);
\draw[thick, gray] (A9C) -- (A9D);
\draw[thick, gray] (A9D) -- (A9E);
\draw[thick, gray] (A9E) -- (A9F);
\draw[thick, gray] (A9F) -- (A9A);
\draw[thick, gray] (A9A) -- (A9C);
\draw[thick, gray] (A9C) -- (A9F);
\draw[thick, gray] (A9F) -- (A9D);

% Hexagon 10 at (5,3)
\coordinate (A10) at (5,3);
\path (A10) ++(0:0.5) coordinate (A10A);
\path (A10) ++(60:0.5) coordinate (A10B);
\path (A10) ++(120:0.5) coordinate (A10C);
\path (A10) ++(180:0.5) coordinate (A10D);
\path (A10) ++(240:0.5) coordinate (A10E);
\path (A10) ++(300:0.5) coordinate (A10F);
\draw[thick] (A10A) -- (A10B);
%\draw[thick] (A10B) -- (A10C);
\draw[thick] (A10C) -- (A10D);
\draw[thick] (A10D) -- (A10E);
\draw[thick] (A10E) -- (A10F);
\draw[thick] (A10F) -- (A10A);
\draw[thick] (A10B) -- (A10F);
\draw[thick] (A10F) -- (A10C);
\draw[thick] (A10C) -- (A10E);

% Hexagon 11 at (0,4)
%%%%%%%%%%%%%%%%%%%%%%%%%%%%%%%%%%%SKIPPED
\coordinate (A11) at (0,4);
\path (A11) ++(0:0.5) coordinate (A11A);
\path (A11) ++(60:0.5) coordinate (A11B);
\path (A11) ++(120:0.5) coordinate (A11C);
\path (A11) ++(180:0.5) coordinate (A11D);
\path (A11) ++(240:0.5) coordinate (A11E);
\path (A11) ++(300:0.5) coordinate (A11F);
\draw[thick] (A11A) -- (A11B);
%\draw[thick] (A11B) -- (A11C);
\draw[thick] (A11C) -- (A11D);
\draw[thick] (A11D) -- (A11E);
\draw[thick] (A11E) -- (A11F);
\draw[thick] (A11F) -- (A11A);
\draw[thick] (A11A) -- (A11E);
\draw[thick] (A11B) -- (A11E);
\draw[thick] (A11C) -- (A11E);

% Hexagon 12 at (-3,5)
\coordinate (A12) at (-3,5);
\path (A12) ++(0:0.5) coordinate (A12A);
\path (A12) ++(60:0.5) coordinate (A12B);
\path (A12) ++(120:0.5) coordinate (A12C);
\path (A12) ++(180:0.5) coordinate (A12D);
\path (A12) ++(240:0.5) coordinate (A12E);
\path (A12) ++(300:0.5) coordinate (A12F);
\draw[thick] (A12A) -- (A12B);
%\draw[thick] (A12B) -- (A12C);
\draw[thick] (A12C) -- (A12D);
\draw[thick] (A12D) -- (A12E);
\draw[thick] (A12E) -- (A12F);
\draw[thick] (A12F) -- (A12A);
\draw[thick] (A12A) -- (A12C);
\draw[thick] (A12A) -- (A12D);
\draw[thick] (A12A) -- (A12E);

% Hexagon 13 at (3,5)
\coordinate (A13) at (3,5);
\path (A13) ++(0:0.5) coordinate (A13A);
\path (A13) ++(60:0.5) coordinate (A13B);
\path (A13) ++(120:0.5) coordinate (A13C);
\path (A13) ++(180:0.5) coordinate (A13D);
\path (A13) ++(240:0.5) coordinate (A13E);
\path (A13) ++(300:0.5) coordinate (A13F);
\draw[thick] (A13A) -- (A13B);
%\draw[thick] (A13B) -- (A13C);
\draw[thick] (A13C) -- (A13D);
\draw[thick] (A13D) -- (A13E);
\draw[thick] (A13E) -- (A13F);
\draw[thick] (A13F) -- (A13A);
\draw[thick] (A13A) -- (A13C);
\draw[thick] (A13C) -- (A13E);
\draw[thick] (A13C) -- (A13F);

% Hexagon 14 at (0,6)
\coordinate (A14) at (0,6);
\path (A14) ++(0:0.5) coordinate (A14A);
\path (A14) ++(60:0.5) coordinate (A14B);
\path (A14) ++(120:0.5) coordinate (A14C);
\path (A14) ++(180:0.5) coordinate (A14D);
\path (A14) ++(240:0.5) coordinate (A14E);
\path (A14) ++(300:0.5) coordinate (A14F);
\draw[thick] (A14A) -- (A14B);
%\draw[thick] (A14B) -- (A14C);
\draw[thick] (A14C) -- (A14D);
\draw[thick] (A14D) -- (A14E);
\draw[thick] (A14E) -- (A14F);
\draw[thick] (A14F) -- (A14A);
\draw[thick] (A14A) -- (A14C);
\draw[thick] (A14C) -- (A14E);
\draw[thick] (A14E) -- (A14A);

% Edges between the triangulations (centers)

\draw[thick, shorten <=15pt, shorten >=15pt] (A1) -- (A2);
\draw[thick, shorten <=15pt, shorten >=15pt] (A1) -- (A3);
\draw[thick, shorten <=15pt, shorten >=15pt] (A1) -- (A4);
\draw[thick, shorten <=15pt, shorten >=15pt] (A2) -- (A5);
\draw[thick, shorten <=15pt, shorten >=15pt] (A2) -- (A6);
\draw[thick, gray, shorten <=15pt, shorten >=15pt] (A3) -- (A9);
\draw[thick, shorten <=15pt, shorten >=15pt] (A3) -- (A10);
\draw[thick, gray, shorten <=15pt, shorten >=15pt] (A4) -- (A7);
\draw[thick, shorten <=15pt, shorten >=15pt] (A4) -- (A8);
\draw[thick, shorten <=15pt, shorten >=15pt] (A5) -- (A8);
\draw[thick, shorten <=15pt, shorten >=15pt] (A5) -- (A11);
\draw[thick, shorten <=15pt, shorten >=15pt] (A6) -- (A10);
\draw[thick, shorten <=15pt, shorten >=15pt] (A6) -- (A11);
\draw[thick, gray, shorten <=15pt, shorten >=15pt] (A7) -- (A9);
\draw[thick, gray, shorten <=15pt, shorten >=15pt] (A7) -- (A12);
\draw[thick, shorten <=15pt, shorten >=15pt] (A8) -- (A12);
\draw[thick, gray, shorten <=15pt, shorten >=15pt] (A9) -- (A13);
\draw[thick, shorten <=15pt, shorten >=15pt] (A10) -- (A13);
\draw[thick, shorten <=15pt, shorten >=15pt] (A11) -- (A14);
\draw[thick, shorten <=15pt, shorten >=15pt] (A12) -- (A14);
\draw[thick, shorten <=15pt, shorten >=15pt] (A13) -- (A14);

\end{tikzpicture}
\end{center}
\end{example}

\begin{example}
Now, using the previous example, we investigate what happens if we contract a mutable variable. Let us start by mutating the variable represented by the dotted diagonal below:\\

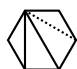
\begin{center}
\begin{tikzpicture}[scale=0.9]
\coordinate (A13) at (3,5);
\path (A13) ++(0:0.5) coordinate (A13A);
\path (A13) ++(60:0.5) coordinate (A13B);
\path (A13) ++(120:0.5) coordinate (A13C);
\path (A13) ++(180:0.5) coordinate (A13D);
\path (A13) ++(240:0.5) coordinate (A13E);
\path (A13) ++(300:0.5) coordinate (A13F);
\draw[thick] (A13A) -- (A13B) -- (A13C) -- (A13D) -- (A13E) -- (A13F) -- cycle;
\draw[thick, densely dotted] (A13A) -- (A13C);
\draw[thick] (A13C) -- (A13E);
\draw[thick] (A13C) -- (A13F);
\end{tikzpicture}
\end{center}

\noindent Note that if $e$ is a mutable variable, then we have
    \begin{align*}
    \mathcal{B}(M/e) &=\Big \{ X \subset E-e : M|_e \textnormal{ has a basis } B \textnormal{ such that } X \cup B \in \mathcal{B}(M) \Big \}\\
    &=\Big \{ X \subset E-e : X \cup e \in \mathcal{B}(M) \Big \}\\
    &=\Big \{ B-e : B \in \mathcal{B}(M) \textnormal{ and } e \in B \Big \}.
    \end{align*}
This will easily imply that the set generating the bases of the new matroid is the one that appears in the following graph:\\
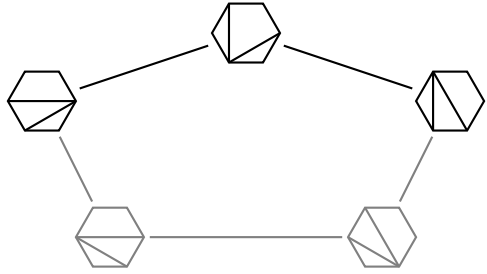
\begin{center}
\begin{tikzpicture}[scale=0.9]

% Hexagon 7 at (-2,4)
%%%%%%%%%%%%%%%%%%%%%%%%%%%%%%%%%%%%%%%%%%%%%%%%%%%%%%%%%%SKIPPED
\coordinate (A7) at (-2,3);
\path (A7) ++(0:0.5) coordinate (A7A);
\path (A7) ++(60:0.5) coordinate (A7B);
\path (A7) ++(120:0.5) coordinate (A7C);
\path (A7) ++(180:0.5) coordinate (A7D);
\path (A7) ++(240:0.5) coordinate (A7E);
\path (A7) ++(300:0.5) coordinate (A7F);
\draw[thick, gray] (A7A) -- (A7B) -- (A7C) -- (A7D) -- (A7E) -- (A7F) -- cycle;
%\draw[thick, gray] (A7A) -- (A7C);
\draw[thick, gray] (A7A) -- (A7D);
\draw[thick, gray] (A7D) -- (A7F);

% Hexagon 9 at (2,3)
\coordinate (A9) at (2,3);
\path (A9) ++(0:0.5) coordinate (A9A);
\path (A9) ++(60:0.5) coordinate (A9B);
\path (A9) ++(120:0.5) coordinate (A9C);
\path (A9) ++(180:0.5) coordinate (A9D);
\path (A9) ++(240:0.5) coordinate (A9E);
\path (A9) ++(300:0.5) coordinate (A9F);
\draw[thick, gray] (A9A) -- (A9B) -- (A9C) -- (A9D) -- (A9E) -- (A9F) -- cycle;
%\draw[thick, gray] (A9A) -- (A9C);
\draw[thick, gray] (A9C) -- (A9F);
\draw[thick, gray] (A9F) -- (A9D);

% Hexagon 12 at (-3,5)
\coordinate (A12) at (-3,5);
\path (A12) ++(0:0.5) coordinate (A12A);
\path (A12) ++(60:0.5) coordinate (A12B);
\path (A12) ++(120:0.5) coordinate (A12C);
\path (A12) ++(180:0.5) coordinate (A12D);
\path (A12) ++(240:0.5) coordinate (A12E);
\path (A12) ++(300:0.5) coordinate (A12F);
\draw[thick] (A12A) -- (A12B) -- (A12C) -- (A12D) -- (A12E) -- (A12F) -- cycle;
%\draw[thick] (A12A) -- (A12C);
\draw[thick] (A12A) -- (A12D);
\draw[thick] (A12A) -- (A12E);

% Hexagon 13 at (3,5)
\coordinate (A13) at (3,5);
\path (A13) ++(0:0.5) coordinate (A13A);
\path (A13) ++(60:0.5) coordinate (A13B);
\path (A13) ++(120:0.5) coordinate (A13C);
\path (A13) ++(180:0.5) coordinate (A13D);
\path (A13) ++(240:0.5) coordinate (A13E);
\path (A13) ++(300:0.5) coordinate (A13F);
\draw[thick] (A13A) -- (A13B) -- (A13C) -- (A13D) -- (A13E) -- (A13F) -- cycle;
%\draw[thick] (A13A) -- (A13C);
\draw[thick] (A13C) -- (A13E);
\draw[thick] (A13C) -- (A13F);

% Hexagon 14 at (0,6)
\coordinate (A14) at (0,6);
\path (A14) ++(0:0.5) coordinate (A14A);
\path (A14) ++(60:0.5) coordinate (A14B);
\path (A14) ++(120:0.5) coordinate (A14C);
\path (A14) ++(180:0.5) coordinate (A14D);
\path (A14) ++(240:0.5) coordinate (A14E);
\path (A14) ++(300:0.5) coordinate (A14F);
\draw[thick] (A14A) -- (A14B) -- (A14C) -- (A14D) -- (A14E) -- (A14F) -- cycle;
%\draw[thick] (A14A) -- (A14C);
\draw[thick] (A14C) -- (A14E);
\draw[thick] (A14E) -- (A14A);

% Edges between the triangulations (centers)

\draw[thick, gray, shorten <=15pt, shorten >=15pt] (A7) -- (A9);
\draw[thick, gray, shorten <=15pt, shorten >=15pt] (A7) -- (A12);
\draw[thick, gray, shorten <=15pt, shorten >=15pt] (A9) -- (A13);
\draw[thick, shorten <=15pt, shorten >=15pt] (A12) -- (A14);
\draw[thick, shorten <=15pt, shorten >=15pt] (A13) -- (A14);

\end{tikzpicture}
\end{center}

Next, we mutate the variable represented by the dotted diagonal below:\\

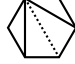
\begin{center}
\begin{tikzpicture}[scale=0.9]
\coordinate (A13) at (3,5);
\path (A13) ++(0:0.5) coordinate (A13A);
\path (A13) ++(60:0.5) coordinate (A13B);
\path (A13) ++(120:0.5) coordinate (A13C);
\path (A13) ++(180:0.5) coordinate (A13D);
\path (A13) ++(240:0.5) coordinate (A13E);
\path (A13) ++(300:0.5) coordinate (A13F);
\draw[thick] (A13A) -- (A13B) -- (A13C) -- (A13D) -- (A13E) -- (A13F) -- cycle;
\draw[thick] (A13A) -- (A13C);
\draw[thick] (A13C) -- (A13E);
\draw[thick, densely dotted] (A13C) -- (A13F);
\end{tikzpicture}
\end{center}

\noindent Using the same logic, it is not hard to see that the graph representing the set that generates the bases is:

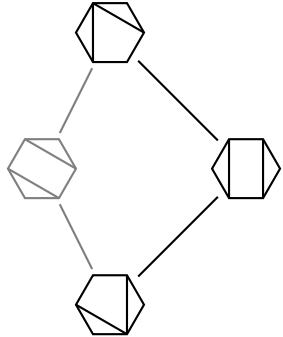
\begin{center}
\begin{tikzpicture}[scale=0.9]

% Hexagon 3 at (3,1)
\coordinate (A3) at (3,1);
\path (A3) ++(0:0.5) coordinate (A3A);
\path (A3) ++(60:0.5) coordinate (A3B);
\path (A3) ++(120:0.5) coordinate (A3C);
\path (A3) ++(180:0.5) coordinate (A3D);
\path (A3) ++(240:0.5) coordinate (A3E);
\path (A3) ++(300:0.5) coordinate (A3F);
\draw[thick] (A3A) -- (A3B) -- (A3C) -- (A3D) -- (A3E) -- (A3F) -- cycle;
\draw[thick] (A3B) -- (A3F);
%\draw[thick] (A13C) -- (A13F);
\draw[thick] (A3D) -- (A3F);

% Hexagon 9 at (2,3)
\coordinate (A9) at (2,3);
\path (A9) ++(0:0.5) coordinate (A9A);
\path (A9) ++(60:0.5) coordinate (A9B);
\path (A9) ++(120:0.5) coordinate (A9C);
\path (A9) ++(180:0.5) coordinate (A9D);
\path (A9) ++(240:0.5) coordinate (A9E);
\path (A9) ++(300:0.5) coordinate (A9F);
\draw[thick, gray] (A9A) -- (A9B) -- (A9C) -- (A9D) -- (A9E) -- (A9F) -- cycle;
\draw[thick, gray] (A9A) -- (A9C);
%\draw[thick] (A13C) -- (A13F);
\draw[thick, gray] (A9F) -- (A9D);

% Hexagon 10 at (5,3)
\coordinate (A10) at (5,3);
\path (A10) ++(0:0.5) coordinate (A10A);
\path (A10) ++(60:0.5) coordinate (A10B);
\path (A10) ++(120:0.5) coordinate (A10C);
\path (A10) ++(180:0.5) coordinate (A10D);
\path (A10) ++(240:0.5) coordinate (A10E);
\path (A10) ++(300:0.5) coordinate (A10F);
\draw[thick] (A10A) -- (A10B) -- (A10C) -- (A10D) -- (A10E) -- (A10F) -- cycle;
\draw[thick] (A10B) -- (A10F);
%\draw[thick] (A13C) -- (A13F);
\draw[thick] (A10C) -- (A10E);

% Hexagon 13 at (3,5)
\coordinate (A13) at (3,5);
\path (A13) ++(0:0.5) coordinate (A13A);
\path (A13) ++(60:0.5) coordinate (A13B);
\path (A13) ++(120:0.5) coordinate (A13C);
\path (A13) ++(180:0.5) coordinate (A13D);
\path (A13) ++(240:0.5) coordinate (A13E);
\path (A13) ++(300:0.5) coordinate (A13F);
\draw[thick] (A13A) -- (A13B) -- (A13C) -- (A13D) -- (A13E) -- (A13F) -- cycle;
\draw[thick] (A13A) -- (A13C);
\draw[thick] (A13C) -- (A13E);
%\draw[thick] (A13C) -- (A13F);

% Edges between the triangulations (centers)

\draw[thick, gray, shorten <=15pt, shorten >=15pt] (A3) -- (A9);
\draw[thick, shorten <=15pt, shorten >=15pt] (A3) -- (A10);
\draw[thick, gray, shorten <=15pt, shorten >=15pt] (A9) -- (A13);
\draw[thick, shorten <=15pt, shorten >=15pt] (A10) -- (A13);

\end{tikzpicture}
\end{center}

Interestingly, one can see here that the mutation of different mutable variables can give two different matroids. This is true because the size in the graph after the first mutation is 5, while it is 4 in the second case.
\end{example}

\begin{remark}
Note that the cluster matroids can be described combinatorially in a nice way, and this description can still be used to describe their minors. However, one can easily see that the class of cluster matroids is not closed under minors. One evidence is the difference of the numbers that we got after contracting two different mutable variables in the previous example.
\end{remark}

%%%%%%%%%%%%%%%%%%%%%%%%%%%%%%%%%%%%%%%%%%%%%%%%%%%%%%%%%%%%%%%%%%%%%%%%%%%%%%%%%%%%%%%%

\section{result summary and remarks}
\begin{remark}
The study of cluster matroids reveals several key structural properties. While n-gon triangulations do not form a set of matroid bases, this limitation is resolved by leveraging the algebraic independence of extended clusters, defining a structure called a cluster matroid. Such a matroid is connected precisely when mutable variables are present. Removing frozen variables from the system results in a uniform matroid, and the Laurent phenomenon enables the construction of a representable matroid from cluster algebra seeds. Notably, deleting or contracting any set of frozen variables always yields another cluster matroid. However, contracting mutable variables does not guarantee analogous consistency, as it may produce distinct matroids. These results highlight the nuanced interplay between frozen and mutable variables in shaping the matroidal framework of cluster structures. The following table summarizes key characteristics and results of this paper.
 \\

\begin{table}[h]
\centering
\caption{Properties of Cluster Matroids}
\label{tab:cluster-matroids}
\begin{tabularx}{\linewidth}{|l|X|}
\hline
\textbf{Question/Property} & \textbf{Answer/Result} \\
\hline
Do $n$-gon triangulations form matroid bases? & No \\ 
\hline
Can this issue be resolved? & Yes, define a \textit{cluster matroid} via algebraic independence \\ 
\hline
Connectedness condition & Connected $\iff$ presence of mutable variables only\\ 
\hline
Effect of removing frozen variables & Yields a uniform matroid \\ 
%\hline
%Constructing representable matroids & Achievable through Laurent phenomenon applied to seeds \\ 
\hline
Deletion/contraction of frozen variables & Always produces another cluster matroid \\ 
\hline
Contraction of mutable variables & Not guaranteed; may generate distinct matroids \\ 
\hline
What about representable matroids? & They can be produced using the Laurent phenomenon \\ 
\hline
\end{tabularx}
\end{table}

\end{remark}
%%%%%%%%%%%%%%%%%%%%%%%%%%%%%%%%%%%%%%%%%%%%%%%%%%%%%%%%%%%%%%%%%%%%%%%%%%%%%%%%%%%%%%%%
\section{Further destination}
This paper is a starting point for a connection between matroids and cluster algebras that can be interestingly deepened. We believe that this can serve as a framework for a new shining path in an interesting area of mathematics that has the potential to grow significantly if it receives proper care from interested mathematicians. It would be very nice to see both theories' simultaneous development in algebraic and combinatorial ways. The work of this paper has the flavor of what matroid theorists usually do. Of course, much deeper work can be built on this, like studying the minors, duals, representability, etc. Other interesting questions arise if we make changes to cluster algebras that build cluster matroids. For example, one might wonder what happens if the ground field of the cluster algebra is changed or if one of its exchange matrix entries is changed. Moreover, the classification of this new class of matroid can be studied and help in answering questions regarding algebraic matroids, which have many unsolved problems yet.\\

It is worth mentioning that this is not the first work that relates both matroid theory and cluster algebras together; there is another interesting path that has a more cluster-algebra flavor. For instance, the reader is referred to the work of Karp and Williams \cite{KW} that is followed by a sequence of other significant papers of the same direction.

\section*{Acknowledgements}
The author thanks Professor James Oxley for the useful discussion in preparing this paper.

%\section*{Disclosure of interest}
%The author declares no competing interests.

%\section*{Funding}
%No funding was received for this work.


\begin{thebibliography}{99}

\bibitem{BFZ}
Berenstein A. $\&$ Fomin S. $\&$ Zelevinsky A.,
\href{https://arxiv.org/pdf/math/0305434.pdf}{\textit{Cluster algebras III. Upper bounds and double Bruhat cells}}, Duke Math. J. \textbf{126} (2005), no. 1, 1-52.

\bibitem{FWZ}
Fomin S. $\&$ Williams L. $\&$ Zelevinsky A.,
\href{https://arxiv.org/pdf/1608.05735.pdf}{\textit{Introduction to Cluster Algebras Chapters 1-3}}, arXiv:1608.05735.

\bibitem{FWZ2}
Fomin S. $\&$ Williams L. $\&$ Zelevinsky A.,
\href{https://arxiv.org/pdf/1707.07190}{\textit{Introduction to Cluster Algebras Chapters 4-5}}, arXiv:1707.07190.

\bibitem {FZ1} Fomin S. $\&$ Zelevinsky A., \href{https://arxiv.org/pdf/math/9802056.pdf}{\textit{Double Bruhat cells and total positivity}}, J. Amer. Math. Soc.
\textbf{12} (1999), 335–380.

\bibitem{FZ}
Fomin S. $\&$ Zelevinsky A.,
\href{https://arxiv.org/pdf/math/0104151.pdf}{\textit{Cluster algebras I. Foundations}}, J. Amer. Math. Soc \textbf{15} (2002), no. 2, 497-529.

\bibitem{FZ2}
Fomin S. $\&$ Zelevinsky A.,
\href{https://arxiv.org/pdf/math/0208229}{\textit{Cluster algebras. II. Finite type classification.}} Invent. Math. \textbf{154}, 1 (2003), 63–121.

\bibitem{KW}
Karp, S. N. and L. K. Williams. \href{https://arxiv.org/pdf/1608.08288}{\textit{The m = 1 amplituhedron and cyclic hyperplane arrangements.}} Int. Math. Res. Not. IMRN 5 (2019): 1401–62.

\bibitem{O}
Oxley J.,
\href{https://www.amazon.com/Matroid-Theory-Oxford-Graduate-Mathematics/dp/0199603391}{\textit{Matroid Theory}}, Second edition, Oxford University Press, New York, 2011.

\end{thebibliography}
\end{document}